\documentclass{amsart}
\usepackage[utf8]{inputenc}
  \usepackage{amsmath}
\usepackage{enumerate}
  \usepackage{amssymb}
  \usepackage{amsthm}
  \usepackage{appendix}
  \usepackage{bm}
  \usepackage{graphicx}
  \usepackage{color}
\usepackage{epsfig}
\usepackage{amsmath,amssymb}
\hoffset - 0.5 in
\newtheorem{Th}{Theorem}[section]
\newtheorem{lem}[Th]{Lemma}

\theoremstyle{definition}

\newtheorem{Cor}[Th]{Corollary}
\newtheorem{Prop}[Th]{Proposition}

\theoremstyle{remark}
\newtheorem*{rem}{\bf Remarks}
\newtheorem*{que}{\bf Questions}

\newtheorem*{thank}{\ \ \ \bf Acknowledgment}

\numberwithin{equation}{section}

\newcommand{\tend}[3][]{\xrightarrow[#2\to#3]{#1}}

\newcommand{\egdef}{\stackrel{\textrm {def}}{=}}
\newcommand{\ds}{\displaystyle}

\newcommand{\N}{\mathbb{N}}

\newcommand{\cl}{\mathcal{L}}

\newcommand{\fr}{\textrm{fr}}

\title[A class of Littlewood polynomials that are not$\cdots$]{A class of Littlewood polynomials that are not $L^\alpha$-flat}

\author{e. H. el Abdalaoui}
\address{Normandy University of Rouen,
  Department of Mathematics, LMRS  UMR 6085 CNRS, Avenue de l'Universit\'e, BP.12,
76801 Saint Etienne du Rouvray - France .}
\email{elhoucein.elabdalaoui@univ-rouen.fr}

\urladdr{http://www.univ-rouen.fr/LMRS/Persopage/Elabdalaoui/}
\address{Department of Mathematics, University of Mumbai, Vidyanagari, Kalina,  Mumbai, 400098, India}
\email{mgnadkarni@gmail.com}
\urladdr{http://www.math.iitb.ac.in/seminar/nadkarni.html}

\date{\today}
\subjclass[2010]{Primary 42A05, 42A55, Secondary 37A05, 37A30}
\dedicatory{with an appendix jointly with M. G. Nadkarni}

\keywords{ merit factor, flat polynomials, ultraflat polynomials, Erd\"{o}s-Newman flatness problem, Littlewood flatness problem, digital transmission, palindromic polynomial, Turyn-Golay's conjecture.
}
\begin{document}
\maketitle
\begin{abstract} We exhibit a class of Littlewood polynomials
that are not $L^\alpha$-flat for any $\alpha \geq 0$.
Indeed, it is shown that the sequence of Littlewood polynomials is not $L^\alpha$-flat, $\alpha \geq 0$,
when the frequency of $-1$ is not in the interval $]\frac14,\frac34[$. We further obtain a generalization of Jensen-Jensen-Hoholdt's result
by establishing that the sequence of Littlewood polynomials is not $L^\alpha$-flat
for any $\alpha> 2$ if the frequency of $-1$ is not $\frac12$.  Finally,
we prove that the sequence of palindromic Littlewood polynomials with even degrees are not $L^\alpha$-flat for any $\alpha \geq 0$.
\end{abstract}

\section{Introduction}\label{intro}
The main goal of this paper is to establish that some class of Littlewood polynomials are not $L^\alpha$-flat, $\alpha \geq 0$.
Precisely, we prove that if the frequency of $-1$ is not in the interval $]\frac14,\frac34[$ or
if the sequence of Littlewood polynomials $(P_q)$ is palindromic with even degrees,
then $(P_q)$ are not $L^\alpha$-flat for any $\alpha \geq 0$.\\

We further establish that the sequence of Littlewood polynomials can not be $L^\alpha$-flat for $\alpha> 2$
if the frequency of $-1$ is not $\frac12$. This strengthen Theorem 2.1 of \cite{Hoholdt-JJ}.\\

It follows that the search for a sequence of $L^\alpha$-flat polynomials from the class $\cl$ can be restricted to the subclass of polynomials $P \in \cl$ which are not palindromic with even degrees and for which the frequency of $-1$ is in the interval $]\frac14,\frac34[$.\\

 The problem of flat polynomials goes back to Erd\"{o}s \cite{ErdosII}, \cite{ErdosIII} and Newman \cite{Newman}. Later, Littlewood asked, in his famous paper \cite{Littlewood}, among several questions, if there are positive absolute constants A and
$B$ such that, for arbitrarily large $n$, one can find a sequence ${\boldsymbol{\epsilon}}=(\epsilon_j)_{j=0}^{n-1} \in \{-1,1\}^n$  such that
$$A \sqrt{n} \leq \Big|\sum_{j=0}^{n-1}\epsilon_jz^j \Big| \leq B \sqrt{n},~~~~~~~~~~~~\forall~z \in S^1.$$
The polynomials of type $\ds L_n(z)\egdef\sum_{j=0}^{n-1}\epsilon_jz^j$ are called nowadays Littlewood polynomials or polynomials from class $\cl$. In the modern terminology, Littlewood's question can be reformulated as follows.
\begin{que}[Littlewood, 1966, \cite{Littlewood}, {\cite[Problem 19]{Littlewood2}}]
  Does there exist a sequence of polynomials from class $\cl$ which is flat in the Littlewood sense?
\end{que}

The  other, more general question, whether there exists a sequence of trigonometric polynomials
 $$K_j(z) = \frac{1}{\sqrt{n_j}}(a_{0,j} + a_{1,j}z + a_{2,j}z^2 + \cdots + a_{q-1,j}z^{n_j-1}), \eqno(2)$$
   $j=1,2,\cdots, \big|a_{k,j}\big| = 1, 0 \leq k \leq n_j-1,$ such that $\big| K_j(z)\big|\rightarrow 1$  uniformly as $j \rightarrow \infty$, was answered affirmatively by J-P Kahane {\cite {Kahane}}. Furthermore, J\'osef Beck {\cite{Beck}} has shown,
   by applying the random procedure of Kahane, that the sequence $K_j, j=1,2,\cdots$ can be chosen to be flat in the sense of Littlewood with  coefficients of $\sqrt {q_j}K_j, j=1,2,\cdots$ chosen from the solutions of $z^{400} =1$. The class of polynomials of type (2) is denoted by $\mathcal{G}$.\\

    For more details on the ultraflat polynomials of Kahane we refer the reader to \cite{QS}. Let us also mention that very recently, Bomberi and Bourgain \cite{B-Bourgain} constructed an effective sequence of ultraflat polynomials.
    \\

 Littlewood's question is also related to the well-know merit factor problem and Turyn-Golay's conjecture \cite{Golay}, arising from digital communications engineering, which states that the merit factor of any binary sequence is bounded.\\

We remind that the merit factor of a binary sequence ${\boldsymbol{\epsilon}}=(\epsilon_j)_{j=0}^{n-1} \in \{-1,1\}^n$ is given by
$$F_n({\boldsymbol{\epsilon}})=\frac{1}{\Big\|P_n\Big\|_4^4-1},$$
where
$P_n(z)=\frac1{\sqrt{n}}\sum_{j=0}^{n-1}\epsilon_j z^j,~~z \in S^1$. Here $S^1$ denotes the circle group.
For a nice account on the merit factor problem, we refer the reader to \cite{Jedwab}, \cite{Borweinal},\cite{Hoholdt}, and for the connection to ergodic theory and spectral theory of dynamical systems to \cite{Down}.\\

The problem of flat polynomials has nowadays a long history and there is a large literature on the subject. Moreover, this problem  is related to some open problems coming from combinatorics, number theory, digital communication, theory of error codes, complex analysis, spectral theory, ergodic theory and others areas.\\

To the best of the author's knowledge, the only general result known on flatness in class $\mathcal{L}$ is due to Saffari and Smith \cite{Saffari1}.
Unfortunately, the authors in \cite{Saffari} point out that their proof contains a mistake. Therefore, the problem remains open. However,
therein, the authors proved that for the palindromic sequence of polynomials from class $\mathcal{L}$ the $L^4$ conjecture of Erd\"{o}s holds (see below).
We shall strengthen their result by proving that the palindromic polynomials with even degrees from class $\mathcal{L}$ are not $L^\alpha$ flat, for $\alpha \geq 0$. This is done by appealing to the Littlewood's criterion \cite{LittlewoodC}.\\



We further exhibit a subclass of Littlewood polynomials which are not $L^\alpha$-flat, $\alpha>0$ by establishing one-to-one correspondance between the Littlewood polynomials and the Newman-Bourgain polynomials given by
$$Q(z)=\frac1{\sqrt{n}}\sum_{j=0}^{n-1} z^{n_j}.$$

Therefore, our main results can be seen as a general results since it reduces the problem of finding flat polynomials in class $\cl$
to a subclass of $\cl$.
Furthermore, it supports the conjecture mentioned by D. J. Newman in \cite{Newman} which says that all the analytic trigonometric polynomials $P$ with coefficients $\pm 1$ satisfy
\begin{eqnarray*}
\big\|P\big\|_1 <c \big\|P\big\|_2,
\end{eqnarray*}
for some positive constant $c<1$. Obviously, this conjecture implies the two  conjectures of Erd\"{o}s's \cite{ErdosI},\cite[Problem 22]{ErdosII}, \cite{ErdosIII} which states that there is a positive constant $d$ such that for any polynomial $P$ from $\cl$ we have
\begin{enumerate}
  \item $\|P\|_4 \geq (1+d)$, ($L^4$ conjecture of Erd\"{o}s.)
  \item  $\|P\|_{\infty} \geq (1+d).$ (Ultraflat conjecture of  Erd\"{o}s.)\\
\end{enumerate}

However, the author in \cite{Ab} proved that the class of Newman-Bourgain polynomials contain a sequence of $L^\alpha$-flat polynomials,
$0<\alpha<2$. This is accomplished by appealing to Singer's construction of the Sidon sets. We refer to \cite{Ab} for more details.\\


This paper is organized as follows. In section 2, we give a brief exposition of some basic tools and we state our main results. In section 3, we prove our first main result in the case that the frequency of $-1$ is not in $[\frac14,\frac34]$. In section 4, we prove our second main result. Finally, in the appendix, we complete the proof of our first main result.

\begin{thank}
The author wishes to express his thanks to Fran\c cois Parreau for many stimulating conversations on the subject and
his sustained interest and encouragement in this work. He is also thankful to  XiangDong Ye and to the University of Sciences and Technology of China where a part of this paper was written, for the invitation and hospitality.
\end{thank}

\section{Basic definitions and tools}
Let $S^1$ denote the circle group and $dz$ the normalized Lebesgue measure on $S^1$. \\
As customary, for $f \in L^1(S^1,dz)$ we define its $n$th Fourier coefficient by
$$\widehat{f}(n)=\int_{S^1} f(z) z^{-n}dz.$$

A polynomial $f(z)=\sum_{j=0}^{n}a_jz^j$ is palindromic if for any $k$, $$\widehat{f}(k)=\widehat{f}(n-k).$$

The $L^2$-normalized Littlewood polynomials are given by
\begin{eqnarray}\label{PfromL}
P_q(z)=\frac1{\sqrt{q}}\sum_{j=0}^{q-1}\epsilon_j z^j,~~~~~~~~~~~~z \in S^1,
\end{eqnarray}
where for each $j=0,\cdots,q-1$, $\epsilon_j \in \{+1,-1\}$.\\

Notice that each sequence  $\epsilon \in \{+1,-1\}^\N$ can be uniquely associated to a sequence in $\eta \in \{0,1\}^\N$ by
putting
$$\epsilon_j=2\eta_j-1,$$ or
$$\epsilon_j=1-2\eta_j', \textrm{~~~with~~}\eta'=\mathbf{1}-\eta,~~~~\mathbf{1}=(1,\cdots,1).$$

The previous remark will play a crucial role in our proof. Indeed, If $(P_q)$ is a sequence of $L^2$-normalized Littlewood polynomials then
\begin{eqnarray}
P_q(z)&=&\frac1{\sqrt{q}}\sum_{j=0}^{q-1}\epsilon_j z^j \nonumber\\
&=&\frac2{\sqrt{q}}\sum_{j=0}^{q-1}\eta_j z^j-\frac1{\sqrt{q}}\sum_{j=0}^{q-1}z^j, \label{Obs1}\\
&=&\frac1{\sqrt{q}}\sum_{j=0}^{q-1}z^j-\frac2{\sqrt{q}}\sum_{j=0}^{q-1}\eta'_j z^j,~~~~~~~~~~~~~~~~~~~z \in S^1.\label{Obs2}
\end{eqnarray}
We put
$$Q_q(z)=\frac1{\sqrt{q}}\sum_{j=0}^{q-1}\eta_j z^j \textrm{~~and~~}~R_q(z)=\frac1{\sqrt{q}}\sum_{j=0}^{q-1}\eta'_j z^j,
~~~~z \in S^1.$$

For a given sequence of Littlewood polynomials $(P_q)$, we may assume without loss of generalities in the sequel that the following limit exist
$$\lim_{ q\longrightarrow +\infty}\frac{\# \Big\{j~~~:~~\widehat{P_q}(j)=-1\Big\}}{q}=\fr(-1)$$
where $\# E$ denote the cardinality of a set $E$. $\fr(-1)$ is the frequency of $-1$ which is also the frequency of $0$ for the sequence of polynomials $(Q_q)$. Note that the frequency of $1$ are the same for the both sequences of polynomials $(P_q)$ and $(Q_q)$.\\

\paragraph{\textbf{A formula between Littlewood and Newman-Bourgain Polynomials.}} We further assume without loss of generalities that the first and last coefficient of $P_q$ are positive in our definition. This makes  the correspondence $T$ defined below one-one. Let $\mathcal{NB}$ denote  the class of  Newman-Bourgain polynomials, i.e., polynomials $Q$ of the type
$$\frac{1}{\sqrt m}(\eta_0 + \eta_1z + \cdots +\eta_{q-2}z^{q-2} + \eta_{q-1}z^{q-1}),$$
where $\eta_0 = \eta_{q-1} =1$, $\eta_i = 0$ or $1, 1\leq i \leq q-2$, and where $m$ is the number of non-zero terms in $Q$ which is also the number of $i$ with $\eta_i =1$. Note that if $P$ is as in \eqref{PfromL} and if we put
$$\eta_i = \frac{1}{2}(\epsilon_i +1), 0\leq i\leq q -1,$$
then the polynomial
$$\frac{1}{\sqrt m}(\eta_0 + \eta_1z + \cdots + \eta_{q-2}z^{q-2} + \eta_{q-1}z^{q-1})$$

is in class $\mathcal{NB}$, where $m$ is the number of $\eta_i = 1$  which is also the number of $\epsilon_i =1$. Let us define one-one invertible map $ T$ from the class $L$ to the class $\mathcal{NB}$ by
\begin{eqnarray*}
(T(P))(z) &=& T\Big(\frac{1}{\sqrt q}\Big(\epsilon_0 + \epsilon_1z + \cdots + \epsilon_{q-2}z^{q-2} + \epsilon_{q-1}z^{q-1}\Big)\Big)\\
&=& \frac{1}{\sqrt m}\Big(\eta_0 + \eta_1z^1 + \cdots + \eta_{q-2}z^{q-2} + \eta_{q-1}z^{q-1}\Big),
\end{eqnarray*}
where $\eta_i = \frac{1}{2}(\epsilon_i +1), 0\leq i\leq q -1$, and $m$ is the  number of $\eta_i =1$ which is
 also the number of $\epsilon_i =1$.\\

Note
$$T^{-1}\Big(\frac{1}{\sqrt m}\Big(\sum_{i=0}^{q-1}\eta_iz^i\Big)\Big) =
\frac{1}{\sqrt q}\Big(\sum_{i=0}^{q-1}(2\eta_i -1)z^i\Big). $$

Let
$$D(z) = D_q(z) = \frac{1}{\sqrt q}\sum_{i=0}^{q-1} z^i .$$
We thus have that $D(1) = \sqrt q$, while for $z \in S^1 \setminus\{1\}$,
$$D(z) = \frac{1}{\sqrt q}\frac{1-z^q}{1-z} \rightarrow 0$$
as $q \rightarrow \infty$.\\

The formula for polynomials in $\mathcal{L}$ mentioned above is as follows: If $P$ is as in \eqref{PfromL}
then

\begin{eqnarray}\label{Tmaps}
P_q(z) &=& 2\frac{\sqrt m}{\sqrt q}(T(P_q))(z) - D(z)\\
&=&2\frac{1}{\sqrt q} A_q(z) - D_q(z),\nonumber
\end{eqnarray}
where $m$ is the number of terms in $P$ with coefficient +1, $A(z) = \sqrt m~~ T(P)(z)$. The proof follows as soon as we write $T(P)(z)$ and $D(z)$ in the right hand side in full form and collect the coefficient of $z^i, ~~0\leq i \leq q-1$.\\

We further define the one-to-one map $S$ from $\cl$ onto $\cl$ by
$$S(P) =\frac1{\sqrt{q}}\Big(\sum_{j=0}^{q-1}(-\epsilon_j)z^j\Big),$$
i.e. the polynomial obtained from $P$ by changing the signs
of $\epsilon_j$, $j=0,\cdots,q-1$. \\

\noindent{}Note that polynomials in $\mathcal{L}$, $\mathcal{NB}$ and the polynomial $D$ all have $L^2(S^1, dz)$ norm $1$.\\

\paragraph{\textbf{Flat polynomials.}}  For any $\alpha>0$ or $\alpha=+\infty$, the sequence $P_n(z), n=1,2,\cdots$ of analytic trigonometric polynomials of $L^2(S^1,dz)$
norm 1 is said to be $L^\alpha$-flat if  $| P_n(z)|, n=1,2,\cdots$ converges in $L^\alpha$-norm to the constant function $1$. For $\alpha=0$,
we say that $(P_n)$ is $L^\alpha$-flat, if the Mahler measures $M(P_n)$ converge to 1. We recall that the Mahler measure of a function $f \in L^1(S^1,dz)$ is defined by
$$ M(f)=\|f\|_0=\lim_{\beta \longrightarrow 0}\|f\|_{\beta}=\exp\Big(\int_{S^1} \log(|f(t)|) dt\Big).$$

The sequence $P_n(z), n=1,2,\cdots$  is said to be flat in almost everywhere sense (a.e. $(dz)$) if $| P_n(z)|$, $n=1,2,\cdots$ converges almost everywhere to $1$ with respect to $dz$.\\

Note that if $P_n(z), n=1,2,\cdots$  is  a.e. $(dz)$ flat then $S(P_n), n=1,2,\cdots$  is also  a.e. $(dz)$  flat.\\

We further say that a sequence $P_n, n=1,2,\cdots$ of polynomials from the class $\mathcal{L}$ (or $\mathcal{G}$) is flat
in the sense of Littlewood  if there exist constants $0<A<B$  such that for all $z\in S^1$, for all $n \in \N$ (or at least for all large $n$)

$$A \leq \big| P_n(z)\big| \leq B.$$

It is obvious that the flatness properties are invariant under $S$. It is further a nice exercise that the $L^4$ conjecture of Erd\"{o}s and the ultraflat conjecture of Erd\"{o}s holds in the class of Newman-Bourgain polynomials.\\

We are now able to state our main results.
\begin{Th}\label{mainL1} Let $(P_q)$ be a sequence of Littlewood polynomials. Suppose that the frequency of
$-1$ is not in the interval $\Big]\frac14,\frac34\Big[$, then
$(P_q)$ are not $L^\alpha$-flat for any $\alpha \geq 0$.
\end{Th}
If we restrict our self to the $L^\alpha$ space with $\alpha>2$, then we have the following much stronger result
\begin{Th}\label{mainL12} Let $(P_q)$ be a sequence of Littlewood polynomials. Suppose that the frequency of
$-1$ is not $\frac12$. Then, the polynomials
$(P_q)$ are not $L^\alpha$-flat for any $\alpha>2$. Furthermore,
$$\lim_{q \longrightarrow +\infty}\Big\|P_q\Big\|_\alpha=+\infty.$$
\end{Th}
We state our second main result as follows.
\begin{Th}\label{mainL2} Let $(P_q)$ be a sequence of Littlewood polynomials. Suppose that each $P_q$ is
palindromic with even degree. Then $(P_q)$ are not $L^\alpha$-flat for any $\alpha \geq 0$.
\end{Th}

\section{Proof of Theorem \ref{mainL1} when the frequency of $-1$ is not in $\Big[\frac14,\frac34\Big].$ }
We start by stating a criterion on the connection between the $L^1$-flatness and $L^\alpha$-flatness, for $\alpha>0.$
\begin{Prop}\label{Key1}Let $\alpha>0$ and $(P_q(z))_{q \geq 0}$ be a sequence of $L^2$-normalized sequence of $L^\alpha$-flat polynomials. Then, there exist a subsequences $(P_{q_n}(z))$ which is almost everywhere flat and $L^1$-flat. Conversely,
Assume that $(P_q(z))_{q \geq 0}$ is $L^1$-flat then there exist a subsequences $(P_{q_n}(z))$ which is almost everywhere flat and $L^\alpha$-flat, for $0<\alpha<2$.
\end{Prop}
For the proof of Proposition \ref{Key1} we need the following tool that is quite useful for proving convergence
 in $L^p$ when the almost everywhere convergence holds without domination.
\begin{Th}[Vitali's convergence theorem ] \label{Vitali}Let $(X,\mathcal{B},\mu)$ be a probability space, $p$ a positive number and $\{f_n\}$ a sequence in $L^p(X)$ which converges in probability to $f$. Then, the following are equivalent:
\begin{enumerate}[(i)]
\item $(|f_n|^p)_{n \geq 0}$ is uniformly integrable;
\item $\ds
\Big|\Big|f_n-f\Big|\Big|_p \tend{n}{+\infty}0.$
\item $\ds \int_X |f_n|^p d\mu \tend{n}{+\infty} \int_X |f|^p d\mu.$
\end{enumerate}
\end{Th}
We remind that the sequence $(f_n)_{n \in \N}$ of integrable functions is said to  be  uniformly integrable  if and only if
$$\int_{\{|f_n|>M\}}\big|f_n\big|(x) d\mu(x) \tend{M}{+\infty}0,$$
uniformly in $n \in \N$. We further have that the condition
\begin{eqnarray*}\label{uniform}
\sup_{ n \in \N}\bigg(\int\big(\big|f_n\big|^{1+\varepsilon}\big)\bigg) < +\infty,
\end{eqnarray*}
for some $\varepsilon$ positive, implies that $\{f_n\}$ are uniformly integrable.\\
\begin{proof}[\textbf{Proof of Proposition \ref{Key1}}]Let $\alpha>0$ and assume that $(P_q(z))_{q \geq 0}$ is $L^\alpha$-flat. Then along a subsequence $(q_n)$ we have that $(|P_{q_n}(z)|)_{n \geq 0}$ converge almost everywhere to $1$. Whence
$(P_{q_n}(z))_{n \geq 0}$ is $L^1$-flat by Vitali's convergence theorem. In the opposite direction, assume that $(P_q(z))_{q \geq 0}$ is $L^1$-flat, then along a subsequence $(|P_{q_n}(z)|)_{n \geq 0}$ converge almost everywhere to $1$. Again by Vitali's convergence theorem, $(P_{q_n}(z))_{n \geq 0}$ is $L^\alpha$-flat for $0 <\alpha<$2.
\end{proof}
In the following we provide a necessary condition for $L^1$-flatness of a sequence of Littlewood polynomials.
\begin{Prop}\label{Key2}Let $(P_q(z))_{q \geq 0}$ be a sequence of $L^2$-normalized  Littlewood polynomials. Suppose that
 $(P_q(z))_{q \geq 0}$ are  $L^1$-flat polynomials, then the frequency of $-1$ is in the interval $\Big[\frac14,\frac34\Big]$.
\end{Prop}
For the proof of Proposition we need the following simple lemma.
\begin{lem}\label{simple}The sequence of polynomials $\ds \Big(\frac1{\sqrt{q}}\sum_{j=0}^{q-1 }z^j\Big)_{q \geq 0}$ is $L^\alpha$-uniformly integrable, for $\alpha \in ]0,2[$.
\end{lem}
\begin{proof}Let $M>0$, $\beta=\frac2{\alpha}$ and $\beta'$ such that $\frac1{\beta}+\frac1{\beta'}=1$. Then, by  H\"{o}lder's inequality, we can write
\begin{eqnarray*}
\int_{\Big\{\Big|\frac1{\sqrt{q}}\sum_{j=0}^{q-1 }z^j\Big|^\alpha>M \Big\}}
\Big|\frac1{\sqrt{q}}\sum_{j=0}^{q-1 }z^j\Big|^\alpha dz &\leq& \Big\|\frac1{\sqrt{q}}\sum_{j=0}^{q-1}z^j\Big\|_2^{\frac2{\beta}} \Big(dz\Big\{\big|\frac1{\sqrt{q}}\sum_{j=0}^{q-1 }z^j\big| > \sqrt[\alpha]{M} \Big\}\Big)^{\frac1{\beta'}}\\
&\leq& \Big(dz\Big\{\big|\frac1{\sqrt{q}}\sum_{j=0}^{q-1 }z^j\big| > \sqrt[\alpha]{M} \Big\}\Big)^{\frac1{\beta'}},
\end{eqnarray*}
since $\ds \Big\|\frac1{\sqrt{q}}\sum_{j=0}^{q-1 }z^j\Big\|_2=1.$ Whence, by Markov inequality , we get
$$dz\Big\{\big|\frac1{\sqrt{q}}\sum_{j=0}^{q-1 }z^j\big| > \sqrt[\alpha]{M} \Big\}
\leq \frac1{\sqrt[\alpha]{M}}\Big\|\frac1{\sqrt{q}}\sum_{j=0}^{q-1 }z^j\Big\|_1 .$$
This gives
$$dz\Big\{\big|\frac1{\sqrt{q}}\sum_{j=0}^{q-1 }z^j\big| > \sqrt[\alpha]{M} \Big\}
\leq \frac1{\sqrt[\alpha]{M}}$$
by Cauchy-Schwarz inequality.
Letting $M \longrightarrow +\infty$, we conclude that
$$\int_{\Big\{\Big|\frac1{\sqrt{q}}\sum_{j=0}^{q-1 }z^j\Big|^\alpha>M \Big\}}
\Big|\frac1{\sqrt{q}}\sum_{j=0}^{q-1 }z^j\Big|^\alpha dz \tend{M}{+\infty} 0.$$
and the proof of the lemma is complete.
\end{proof}
\begin{proof}[\textbf{Proof of Proposition \ref{Key2}}]By $\eqref{Obs2}$, we have
\begin{eqnarray*}
P_q(z)&=&\frac1{\sqrt{q}}\sum_{j=0}^{q-1}z^j-2R_q(z),~~~~~~~~~~\forall z \in S^1.
\end{eqnarray*}
We further have, for any $z \neq1$,
$$\Big|\frac1{\sqrt{q}}\sum_{j=0}^{q-1 }z^j\Big| \tend{q}{+\infty}0.$$
Hence
$$\Big\|\frac1{\sqrt{q}}\sum_{j=0}^{q-1 }z^j\Big\|_1 \tend{q}{+\infty}0,$$
by Vitali's convergence theorem. Therefore
$$\Big\|\big|P_q(z)\big|-\big|2R_q(z)\big|\Big\|_1 \tend{q}{+\infty}0.$$
It follows that $(P_q(z))_{q \geq 0}$ is $L^1$-flat if and only if
$$\Big\|\big|R_q(z)\big|-\frac12\Big\|_1 \tend{q}{+\infty}0.$$
Assuming that $(P_q(z))_{q \geq 0}$ is $L^1$-flat. It follows that we have
$$\Big\|R_q(z)\Big\|_1 \tend{q}{+\infty}\frac12.$$
Whence, by Cauchy-Schwarz inequality, we can write
$$\Big\|R_q(z)\Big\|_2=\sqrt{\frac{\# \Big\{j:\widehat{R_q}(j)=1\Big\}}{q}} \geq \Big\|R_q(z)\Big\|_1.$$
Letting $q\longrightarrow +\infty$, we obtain
\begin{eqnarray}\label{eq:1}
\lim_{q \longrightarrow +\infty}\frac{\# \Big\{j:\widehat{R_q}(j)=1\Big\}}{q}= \fr({-1}) \geq \frac14.
\end{eqnarray}
We now apply this arguments again, with $R_q$ replaced by $Q_q$, to obtain
\begin{eqnarray}\label{eq:2}
\fr(1)=1-{\fr({-1})} \geq \frac1{4}.
\end{eqnarray}
The proof of the lemma follows by Combining \eqref{eq:1} with \eqref{eq:2}.
\end{proof}
At this point, we conclude that the proof of the main result (Theorem \ref{mainL1}), when the frequency of $-1$ is not in $[\frac14,\frac34]$, follows easily from Proposition \ref{Key2}.\\

From Lemma \ref{simple} it is a simple matter to strengthen Proposition \ref{Key1} as follows
\begin{Prop}\label{Key3}Let $(P_q(z))_{q \geq 0}$ be a sequence of $L^2$-normalized Littlewood polynomials. Suppose that
 $(P_q(z))_{q \geq 0}$ are $L^\alpha$-flat polynomials, $0<\alpha<2$, then the frequency of negative coefficients $\fr(-1)$ verify
$$\frac14 \leq \fr(-1) \leq \frac34,$$
\end{Prop}
Proposition \ref{Key3} is related to the following theorem due to Jensen-Jensen and H{\o}holdt  \cite{Hoholdt-JJ}.
\begin{Th}\label{JJH}Let $(P_q(z))_{q \geq 0}$ be a sequence of  $L^2$-normalized Littlewood polynomials. Suppose that
$$\frac{\#\big\{~~j~~: \widehat{P_q}(j)=-1\big\}}{q} \longrightarrow \fr(-1)$$ as $q\longrightarrow +\infty$. If $\fr(-1) \neq \frac12$ then
$$\Big\|P_q\Big\|_4 \tend{q}{+\infty}+\infty.$$
\end{Th}

Obviously, Theorem \ref{JJH} will follow immediately form our main result Theorem \ref{mainL12}. Let us give its proof
\begin{proof}[\textbf{Proof of Theorem \ref{mainL12}}]
 Let $\beta=\frac{\alpha}{2}$, and apply Marcinkiewicz-Zygmund inequalities to get
 \begin{eqnarray}\label{MZj}
  \Big\|\big|P_q\big|^2-1\Big\|_\beta^{\beta} \geq \frac{A_\beta}{q}\Big|\Big|\big|P_q(1)\big|^2-1\Big|^\beta.
 \end{eqnarray}
We further have
\begin{eqnarray*}
\big|P_q(1)\big|^2 &=&\Big|\sqrt{q}-2\frac{n_q}{\sqrt{q}}\Big|^2,
\end{eqnarray*}
where $n_q$ is the number of $\eta_j'=1$ which is the number of $\epsilon_j=-1$. This equality is due to the fact that $P_q(z)=\frac1{\sqrt{q}}\sum_{j=0}^{q-1}z^j-2R_q(z),$
and $\ds R_q(1)=\frac{n_q}{\sqrt{q}}=\frac{\# \Pi_q}{\sqrt{q}}.$ From this, we can write

\begin{eqnarray*}
\big|P_q(1)\big|^2&=&q\Big(1-2\frac{n_q}{q}\Big)^2.
\end{eqnarray*}
Whence
\begin{eqnarray*}
  \Big\|\big|P_q\big|^2-1\Big\|_\beta^{\beta}  &\geq& \frac{A_\beta}{q} \Big|q\Big(1-2\frac{n_q}{q}\Big)^2-1\Big|^\beta\\
&\geq& A_\beta \Big|\Big(1-2\frac{n_q}{q}\Big)^2-\frac1{q}\Big|^\beta q^{\beta-1}.
 \end{eqnarray*}
Therefore, by the triangle inequality, we can rewrite \eqref{MZj} as follows
\begin{eqnarray*}
  \Big(\Big\|P_q\Big\|_\alpha^{2}+1\Big)^{\beta} \geq A_\beta \Big|\Big(1-2\frac{n_q}{q}\Big)^2-\frac1{q}\Big|^\beta q^{\beta-1}.
 \end{eqnarray*}
Letting $q\longrightarrow +\infty$, we conclude that
$$\lim_{q\longrightarrow +\infty} \Big(\|P_q\|_\alpha^{2}+1\Big)^{\beta} \geq \big(1-2\fr(-1)\big)^\alpha \lim_{q\longrightarrow +\infty}q^{\beta-1}=+\infty,$$
since $\fr(-1) \neq \frac12$. The proof of the theorem is complete.
 \end{proof}

It follows from our proof that if the sequence of polynomials $P_n$, $n=1,\cdots,$ from the class $\mathcal{L}$ is flat in the
Littlewood sense then the frequency of -1 is $\frac12$.

\section{Proof of Theorem \ref{mainL2}.}
The main tool in the proof of our second main result is the following Littlewood's criterion of flatness.
\begin{Th}[Littlewood's criterion \cite{LittlewoodC}] Let $\ds f_n(t)=\sum_{j=1}^{n}a_m \cos(m t+\phi_m)$ and assume that we have
$$\sum_{m=1}^{n}a_m^2 \leq \frac{K}{n^2} \sum_{m=1}^{n}m^2a_m^2,$$
for some absolute constant $K$. Then, for any $\alpha>0$ there exists a constant $A(k,\alpha)$ such that
$$
  \begin{array}{ll}
    \|f_n\|_\alpha \leq \big(1-A(k,\alpha)\big)\|f_n\|_2, & \hbox{ if $\alpha<2$;} \\
    \|f_n\|_\alpha \geq \big(1+A(k,\alpha)\big)\|f_n\|_2, & \hbox{ if $\alpha>2$.}
  \end{array}
$$
\end{Th}
Notice that we have
$$\big\|f_n'\big\|_{2} \leq n \big\|f_n\big\|_2,$$
 by Bernstein-Zygmund inequalities \cite[Theorem 3.13, Chapter X, p. 11]{Zygmund}. Furthermore,
the assumption in the Littlewood's criterion  say that there is a constant $K$ such that
$$\big\|f_n'\big\|_{2} \geq K n \big\|f_n\big\|_2.$$

We proceed now to prove our second main result.\\

Let $(P_n(z))$ be a sequence of palindromic polynomials from class $\mathcal{L}$ and put
 $$P_n(z)=\sum_{j=0}^{n}\epsilon_j z^j,~~~~~~~~~~n=2,4,6\cdots, ~~~z \in S^1.$$

A straightforward computation gives
$$P_n(z)=z^{\frac{n}{2}}L_n(z)-\epsilon_{\frac{n}{2}}z^{\frac{n}{2}},~~~~~~~~~~~~~\forall z \in S^1,$$
where
$$L_n(z)=\sum_{k=0}^{\frac{n}{2}}\epsilon_k \sigma_{\frac{n}{2}-k}(z),$$
and
$$ \sigma_{l}(z)=z^l+\frac1{z^l}.$$
Therefore, for any $z \in S^1$, we have
$$L_n(\theta)=\sum_{k=0}^{\frac{n}{2}}a_{\frac{n}{2}-k}\cos(k\theta),~~~~~~~~~~a_k=2\epsilon_k.$$
Applying the Littlewood criterion, it follows that $L_n$ are not $L^\alpha$-flat, $\alpha \geq 0$. We thus conclude that
$(P_n)$ are not  $L^\alpha$-flat, $\alpha \geq 0$. This finish the proof of our second main result.

\section{Appendix. \small{Proof of Theorem \ref{mainL1} when the frequency of $-1$ is $1/4$ or $3/4$.}\\ jointly with M. G. Nadkarni}
 For the proof of our first main result when the frequency of $-1$ is equal to $\frac14$ or $\frac34$, we need some tools from \cite{Abd-Nad}. \\
 Let $Q(z) = \ds \frac{1}{\sqrt m}\sum_{j=0}^{q-1}\eta_jz^j$ be a polynomial in the class $\mathcal{NB}$, where $m = \ds \sum_{j=0}^{q-1}\eta_j$, which is the number of nonzero terms in $Q$. Note that $Q(1) = \sqrt m$.

 $$\mid Q(z)\mid^2 = 1 + \sum_{\overset{k=-q-1}{k\neq 0}}^{q-1}a_kz^{k} ,$$
where each  $a_k$ is a sum of terms of the type $\eta_i\eta_j\frac{1}{m}, i\neq j$. Note that for each $k$, $a_{-k} = a_k$.
Write
$$L = \sum_{\overset{j=-(q -1)}{j\neq 0}}^{q-1}a_j =\mid  Q(1)\mid^2 -1 = m-1.$$
Consider the random variables  $X(k) =z^{k} - a_k, -(q-1)\leq k \leq q-1$ with respect to the measure
$\nu = \big|Q(z)\big|^2dz$. We write $m(k,l) = \int_{S^1}X(k)\overline{X(l)}d\nu$, $-(q-1) \leq k,l\leq q-1, k, l \neq 0$ and $M$ for
the covariance  matrix with entries $m(k,l), -(q-1) \leq k,l\leq q-1, k, l \neq 0$. We call $M$ the covariance matrix
associated to $\nu =\mid Q(z)\mid^2dz$. Since linear combination of $X(k), -(q-1) \leq k\leq q-1, k\neq 0$, can vanish at no more
than a finite set in $S^1$, and, $\nu$ is non discrete, the random variables $X(k), -(q-1) \leq k \leq q-1, k\neq 0$ are linearly
independent, whence the covariance matrix $M$ is non-singular. $M$ is a $2(q-1)\times 2(q-1)$ positive definite matrix.\\

Note that

$$m_{i,j} =\int_{S^1}z^{i -j}d\nu  - a_ia_j,~~~~ m_{i,i} = 1 -  a_i^2$$

Let $r(Q) = r$ denote the sum of the entries of the matrix $M$, $r$ is positive.  Let
$C(Q) = C =$ sum of the absolute values of the entries of $M$. Note that $ r \leq C$. Also note that since
each $\big|m_{i,j}\big| \leq 1$ we have
$$C \leq (2q-1)^2 <4q^2.$$

We will now consider a sequence $Q_n(z),$ $j=1,2,\cdots$ of polynomials from the class $\mathcal{NB}$. The quantities $M(Q_n)$, $r(Q_n), C(Q_n)$ etc will now written as $M_n, r_n, C_n$ etc.\\

The following theorem is a special case of Theorem 5.1 in \cite{Abd-Nad}.
\begin{Th}\label{elab-nad}
If $Q_n(z)$, $n=1,2,\cdots,$ is an $a.e.~~(dz)$ flat sequence from the class $\mathcal{NB}$, then
$$\frac{C_n}{m_n^2} \tend{n}{+\infty}+\infty.$$
\end{Th}
We thus get
\begin{Cor}\label{elab-nad2}If $Q_n$ ,$n =1,2,\cdots$ is an $a.e.~~(dz)$ flat sequence then the
ratios $\frac{m_n}{q_n}$, $n =1,2,\cdots$ converge to zero.
\end{Cor}
\begin{proof} If not there is a subsequence over which the ratios $\frac{m_n}{q_n}$, $n =1,2,\cdots$ converge to a positive constant
$c \leq 1$. We may assume without loss of generality that $\Big(\frac{m_n}{q_n}\Big)$ converge to $c$. Since
$C_n \leq 4 q_n^2$, $n = 1,2,\cdots$, we conclude that
$$\frac{C_n}{m_n^2} \leq \frac{4q_n^2}{m_n^2}\tend{n}{+\infty}\frac4{c^2}<+\infty,$$
which is a contradiction. The corollary follows.
\end{proof}
We proceed now by contradiction to complete the proof of our first main result.\\
Assume that $P_n$ is $a.e.~~ (dz)$ flat and the frequency of $-1$ is $\frac14$. Then,
$S(P_n)$, $n = 1,2,\cdots,$ is also flat in $a.e.~~ (dz)$ sense. We further have that
the frequency of $1$ is  $\frac14$, i.e., $\frac{m_n}{q_n}\longrightarrow \frac14$ as $n\longrightarrow \infty$. We thus get,
 by the formula \eqref{Tmaps},  $T(S(P_n))=Q_n$ is $a.e.~~ (dz)$ flat sequence in $\mathcal{NB}$ with
 $\frac{m_n}{q_n}\tend{n}{\infty}\frac14.$ which is
 impossible by Corollary \ref{elab-nad2}. In the same manner, we can see that the same conclusion hold for $\lim_{n \longrightarrow +\infty}\frac{m_n}{q_n}=\frac3{4},$ by appealing to the formula $(2.3)$. This finish the proof of our first main result.
\begin{rem}$\quad$\\
 1) Formula \eqref{Tmaps} at once shows that if a sequence $P_n, n=1,2,\cdots$ in the class $\mathcal{L}$ is ultraflat
 then
\begin{enumerate}[(i)]
 \item   $\ds \lim_{n \rightarrow \infty}\frac{m_n}{q_n} = \frac{1}{2}$ and
 \item $T(P_n), n=1,2,\cdots$ converges uniformly to $\frac{1}{\sqrt 2}$ on compact subsets of $S^1\setminus\{1\}.$
\end{enumerate}
It is not known if (i) and (ii) are compatible conditions. However, the numerical evidence from \cite{odlyzko} suggest that (i) and (ii) are not compatible.\\

\noindent{}2) Exploring the limit distribution of the sequence of polynomials from class $\mathcal{L}$ can be linked to the exploration of the the limit distribution of the sequence of polynomials from class $\mathcal{NB}$ by \eqref{Tmaps}.  Characterization of the class of distributions which can be a limit distribution  of  a sequence of polynomials from class $\mathcal {N B}$ is an open problem. For a very recent work on the subject, we refer to \cite{fukuyama} and the references therein.
\end{rem}

  \end{document}